\documentclass[10pt]{amsart}
\usepackage{a4}
\usepackage{amssymb}
\usepackage{array}
\usepackage{booktabs}
\usepackage{multirow}
\usepackage{hhline}
\usepackage{comment}
\usepackage{color}
\usepackage{amsmath,amsthm}
\usepackage{amssymb,latexsym}
\usepackage{enumerate}

\newtheorem{theorem}{Theorem}[section]
\newtheorem{lemma}[theorem]{Lemma}
\newtheorem{proposition}[theorem]{Proposition}

\theoremstyle{definition}

\theoremstyle{remark}
\newtheorem{remark}[theorem]{Remark}

\numberwithin{equation}{section}

\def\mconv{\stackrel{\mathrm m}{\longrightarrow}}

\begin{document}
\title{Asymptotic Spectral Distributions of Distance $k$-Graphs of Cartesian Product Graphs}
\author{Yuji Hibino}
\author{Hun Hee Lee}
\author{Nobuaki Obata}

\address{Yuji Hibino:
Department of Mathematics, Saga University,
Saga, 840-8502, Japan}
\email{hibinoy@cc.saga-u.ac.jp}

\address{Hun Hee Lee:
Department of Mathematical Sciences,
Seoul National University,
San56-1 Shinrim-dong Kwanak-gu
Seoul 151-747, Republic of Korea}
\email{hunheelee@snu.ac.kr}

\address{Nobuaki Obata:
Graduate School of Information Sciences, Tohoku University,
Sendai, 980-8579, Japan}
\email{obata@math.is.tohoku.ac.jp}


\begin{abstract}
Let $G$ be a finite connected graph on two or more vertices
and $G^{[N,k]}$ the distance $k$-graph of the $N$-fold Cartesian power of $G$.
For a fixed $k\ge1$, we obtain explicitly the large $N$ limit of 
the spectral distribution (the eigenvalue distribution of the adjacency matrix) of $G^{[N,k]}$.
The limit distribution is described in terms of the Hermite polynomials.
The proof is based on asymptotic combinatorics along with quantum probability theory. 
\end{abstract}

\maketitle

\section{Introduction}

Since Vershik \cite{Vershik95} emphasized the importance of 
asymptotic problems in combinatorics, various approaches have been developed 
from different branches of mathematics. 
The main question in this context is to explore the limit behavior of 
a combinatorial object when it grows.
\textit{Asymptotic spectral analysis} of a growing graph is a subject in this line
with wide applications to structural analysis of complex networks.

In this paper, we study a particular class of growing graphs naturally induced
from the Cartesian powers of a finite connected graph. 
In fact, we will prove the following main result.

\begin{theorem}\label{1thm:main result}
Let $G=(V,E)$ be a finite connected graph with $|V|\ge2$.
For $N\ge1$ and $k\ge1$ let $G^{[N,k]}$ be the distance $k$-graph of $G^N=G\times \cdots\times G$ 
($N$-fold Cartesian power)
and $A^{[N,k]}$ its adjacency matrix.
Then, for a fixed $k\ge1$, 
the eigenvalue distribution of $N^{-k/2}A^{[N,k]}$ converges in moments as $N\rightarrow\infty$
to the probability distribution of
\begin{equation}\label{1eqn:main limit}
\left(\frac{2|E|}{|V|}\right)^{k/2}
\frac{1}{k!}\Tilde{H}_k(g),
\end{equation}
where $\Tilde{H}_k$ is the monic Hermite polynomial of degree $k$ (see Section~\ref{subsec:Main Result})
and $g$ is a random variable obeying the standard normal distribution $N(0,1)$.
\end{theorem}

It is noteworthy that the limit distribution \eqref{1eqn:main limit} is 
obtained explicitly and is universal 
in the sense that it is independent of the details of a factor $G$.
Namely, for a large $N$, spectral structure of the distance $k$-graph of $G^N$ is
dominated by the product structure.
This shares a common nature with the central limit theorem in probability theory.
In fact, we will prove the above result along with quantum 
(noncommutative) probability theory \cite{HO-BOOK},
where central limit theorems of various kinds
have been studied from algebraic and combinatorial viewpoints.

The study of asymptotic spectral distribution of $G^{[N,k]}$ for a large $N$ limit
appeared first in \cite{Kurihara-Hibino2011} where
the case of $G=K_2$ (the complete graph on two vertices) and $k=2$ was studied
by means of quantum decomposition.
Later in \cite{Obata2012} 
the spectrum of the distance $k$-graph of $H(N,2)=K_2^N$
is explicitly obtained in terms of the Krawtchouk polynomials for arbitrary $1\le k\le N$.
Then, by using certain limit formulas of the Krawtchouk polynomials, 
the asymptotic spectral distribution of the distance $k$-graph of $H(N,2)$ is determined.
The result is a special case of \eqref{1eqn:main limit} with $|V|=2$ and $|E|=1$.
In the recent paper \cite{Hibino} the above argument  
is extended to cover the distance $k$-graph of the Hamming graph $H(N,d)=K_d^N$.
The result is again a special case of \eqref{1eqn:main limit}.
While, the case of $G$ being a star graph and $k=2$ is discussed in \cite{Kurihara2012}.
During these studies it has been conjectured that 
the limit distribution does not depend on the detailed structure of $G$,
as the central limit distribution of the sum of 
independent, identically distributed random variables is the normal (Gaussian) law
independently of the distributions of the random variables. 
Our main result shows that this conjecture is true. 

This paper is organized as follows.
In Section \ref{sec:Preliminaries}, recalling some notions and notations in
quantum probability, 
we prepare a useful result on the convergence of algebraic random variables
(Proposition~\ref{1prop:basic convergence result})
and reformulate our main result (Theorem~\ref{3thm:main result}).
In Section \ref{2sec:Convergence} we derive a combinatorial limit formula (Theorem~\ref{2thm:convergence of B[n]}),
which is viewed as an extension of the commutative central limit theorem in quantum probability.
In Section \ref{sec:distance k-graphs} we prove the main result.
Our discussion is based on asymptotic estimation of combinatorial objects,
along with the philosophy of Vershik \cite{Vershik95}.

Finally, we mention some relevant works.
The distance $k$-graphs are introduced originally in the study of distance-regular graphs,
see e.g., \cite{BCN,Fiol1997}.
The adjacency matrix of the distance $k$-graph of a finite graph $G$, say $D^{[k]}=A^{[1,k]}$, is 
nothing else but the $k$-distance matrix of $G$.
Then the distance matrix $D$ of $G$ is defined by
\[
D=\sum_{k=1}^\infty k D^{[k]},
\]
where the right-hand side is in fact a finite sum.
The spectrum of the distance matrix has been actively studied recently, 
in particular, in connection with spectral graph theory, see e.g., \cite{Dalfo} and references cited therein.
Asymptotic spectral analysis of the distance matrix will be an interesting research topic in this connection.
Distance $k$-graphs are used to construct embeddings of graphs into
metric spaces for measuring graph similarity which has wide applications 
in statistical pattern recognition \cite{Czech}.
The asymptotic spectral analysis, being related to the graph embeddings,
is expected to contribute some applications in this line of research.
It is also noteworthy that the probability distribution \eqref{1eqn:main limit} is derived
by Hora \cite{Hora} from the asymptotic behavior of the Young graph 
(branching rule of representations of the symmetric groups).

\section{Preliminaries}
\label{sec:Preliminaries}

\subsection{Algebraic Probability Space}

An \textit{algebraic probability space} is a pair $(\mathcal{A},\varphi)$,
where $\mathcal{A}$ is a $*$-algebra over the complex number field $\mathbb{C}$
with multiplication identity $1=1_{\mathcal{A}}$
and $\varphi$ a state on it, i.e., $\varphi:\mathcal{A}\rightarrow\mathbb{C}$ is a $\mathbb{C}$-linear 
function on $\mathcal{A}$ satisfying
$\varphi(1)=1$ and $\varphi(a^*a)\ge0$ for all $a\in\mathcal{A}$.
We do not assume any topological conditions.
An element $a\in\mathcal{A}$ is called an \textit{(algebraic) random variable} and
it is called \textit{real} if $a^*=a$.
It is known that 
\[
\varphi(a^*)=\overline{\varphi(a)},
\qquad a\in\mathcal{A},
\]
and the Schwarz inequality holds:
\[
|\varphi(a^*b)|^2\le \varphi(a^*a)\varphi(b^*b),
\qquad a,b\in\mathcal{A}.
\]
In particular, for real random variables $a=a^*$ and $b=b^*$ we have
\[
|\varphi(ab)|^2\le \varphi(a^2)\varphi(b^2).
\]
A state $\varphi$ is called \textit{tracial} if
\[
\varphi(ab)=\varphi(ba),
\qquad
a,b\in \mathcal{A}.
\]

For a real random variable $a\in\mathcal{A}$ there exists a probability distribution $\mu$
on the real line $(-\infty,+\infty)$ such that
\[
\varphi(a^m)=\int_{-\infty}^{+\infty} x^m\mu(dx),
\qquad
m=1,2,\dots.
\]
The above $\mu$ is called the \textit{spectral distribution} of $a$ in the state $\varphi$.
The existence of $\mu$ follows from the Hamburger theorem; 
however, the uniqueness does not hold
in general due to the famous indeterminate moment problem,
for further details see \cite[Chapter~1]{HO-BOOK}.

\subsection{Convergence in Moments}

Let $(\mathcal{A}_n,\varphi_n)$, $n=1,2,\dots$, and $(\mathcal{A},\varphi)$ be algebraic probability spaces.
We say that a sequence of real random variables $a_n\in\mathcal{A}_n$ converges to 
a real random variable $a\in\mathcal{A}$ in moments if
\[
\lim_{n\rightarrow\infty}\varphi_n(a_n^m)=\varphi(a^m),
\qquad m=1,2,\dots.
\]
In this case we write
\[
a_n \mconv a
\]
for simplicity.

\begin{proposition}
If $a_n \mconv a$, then for any polynomial $p(x)$ we have
$p(a_n) \mconv p(a)$.
\end{proposition}

The above assertion is obvious. However, generalization to
multivariable case is not trivial.
For real random variables $a,b,\dots,c\in\mathcal{A}$ the quantities of the form:
\[
\varphi(a^{\alpha_1}b^{\beta_1}\dotsb c^{\gamma_1}
\dotsb a^{\alpha_i}b^{\beta_i}\dotsb c^{\gamma_i}\dotsb),
\]
where $\alpha_i,\beta_i\dots,\gamma_i$ are non-negative integers,
are called \textit{$m$-th mixed moments} of $a,b,\dots,c\in\mathcal{A}$ with 
$m=\sum_i(\alpha_i+\beta_i+\dotsb+\gamma_i)$.

\begin{proposition}\label{1prop:basic convergence result}
Let $(\mathcal{A}_n,\varphi_n)$, $n=1,2,\dots$, and $(\mathcal{A},\varphi)$ be algebraic probability spaces.
Let $k\ge1$ be a fixed integer.
Let $a_n=a_n^*,z_{1n}=z_{1n}^*,\dots, z_{kn}=z_{kn}^* \in\mathcal{A}_n$,
$n=1,2,\dots$, and $a=a^*\in\mathcal{A}$ be real random variables,
and $\zeta_1,\dots,\zeta_k\in\mathbb{R}$.
Assume the following conditions:
\begin{enumerate}
\item[\upshape (i)] $a_n \mconv a $ and $z_{in}\mconv \zeta_i 1$ for $i=1,2,\dots,k$;
\item[\upshape (ii)] $\varphi_n$ is a tracial state for $n=1,2,\dots$;
\item[\upshape (iii)] $\{a_n, z_{1n}\,, \dots, z_{kn}\}\subset\mathcal{A}_n$ have uniformly bounded mixed moments in the sense that
\begin{align*}
C_m
=\sup_n \max
&\bigg\{|\varphi_n(a_n^{\alpha_1}z_{1n}^{\beta_1}\dotsb z_{kn}^{\delta_1}
\dotsb a_n^{\alpha_i}z_{1n}^{\beta_i}\dotsb z_{kn}^{\delta_i}\dotsb)|\,; \\
&\qquad\qquad\qquad
\begin{array}{l}
\text{$\alpha_i,\beta_i,\dots,\delta_i\ge0$ are integers} \\
\sum_i(\alpha_i+\beta_i+\dotsb+\delta_i)=m
\end{array}
\bigg\}<\infty.
\end{align*}
\end{enumerate}
Then, for any non-commutative polynomial $p(x,y_1,\dots,y_k)$ we have
\begin{equation}\label{1eqn:convergence in moments (300)}
p(a_n,z_{1n},\dots,z_{kn}) \mconv p(a,\zeta_11,\dots,\zeta_k1).
\end{equation}
\end{proposition}

\begin{remark}
Strictly speaking,
\eqref{1eqn:convergence in moments (300)} is abuse of notation  
because $p(a_n,z_{1n},\dots,z_{kn})$ is not necessarily real.
We understand tacitly \eqref{1eqn:convergence in moments (300)} to be
\[
\lim_{n\rightarrow\infty}\varphi_n(p(a_n,z_{1n},\dots,z_{kn})^m)
=\varphi(p(a,\zeta_11,\dots,\zeta_k1)^m),
\qquad m=1,2,\dots.
\]
\end{remark}

\begin{remark}
Obviously, condition (ii) in Proposition~\ref{1prop:basic convergence result}
may be replaced with
\begin{enumerate}
\item[(ii$^\prime$)] $\varphi_n$ restricted to the $*$-subalgebra generated by
$\{a_n, z_{1n}\,, \dots, z_{kn}\}$ is tracial.
\end{enumerate}
Then we note that if $\{a_n, z_{1n}\,, \dots, z_{kn}\}$ are mutually commutative, 
conditions (ii) and (iii) are redundant.
In fact, as condition (ii$^\prime$) is trivially satisfied (ii) is redundant.
For condition (iii) we first observe that
\begin{equation}\label{1eqn:in remark}
|\varphi_n(a_n^{\alpha_1}z_{1n}^{\beta_1}\dotsb z_{kn}^{\delta_1}
\dotsb a_n^{\alpha_i}z_{1n}^{\beta_i}\dotsb z_{kn}^{\delta_i}\dotsb)|
=|\varphi_n(a_n^{\alpha}z_{1n}^{\beta}\dotsb z_{kn}^{\delta})|,
\end{equation}
where $\alpha=\sum_i\alpha_i$, $\beta=\sum_i\beta_i$, $\dots$, $\delta=\sum_i\delta_i$.
Applying the Schwarz inequality repeatedly we have
\begin{align*}
|\varphi_n(a_n^{\alpha}z_{1n}^{\beta}z_{2n}^{\gamma}\dotsb z_{kn}^{\delta})|^2
&\le \varphi_n(a_n^{2\alpha})\varphi_n(z_{1n}^{2\beta}z_{2n}^{2\gamma}\dotsb z_{kn}^{2\delta}), \\
|\varphi_n(z_{1n}^{2\beta}z_{2n}^{2\gamma}\dotsb z_{kn}^{2\delta})|^2
&\le \varphi_n(z_{1n}^{4\beta})\varphi_n(z_{2n}^{4\gamma} \dotsb z_{kn}^{4\delta}), \\
& \dotsb. 
\end{align*}
Finally, \eqref{1eqn:in remark} is bounded by a product of moments of $a_n\,,z_{1n}\,,\dots, z_{kn}$,
which remain finite as $n\rightarrow\infty$ since they are convergent sequences by condition (i).
Thus, condition (iii) holds.
\end{remark}

\begin{proof}[Proof of Proposition~\ref{1prop:basic convergence result}]
Since $p(a_n,z_{1n},\dots,z_{kn})^m$, $m\ge1$, is again a non-commutative polynomial in
$a_n,z_{1n},\dots,z_{kn}$, it is sufficient to prove that
\begin{equation}\label{1eqn:in proof 1.1 (100)}
\lim_{n\rightarrow\infty}
\varphi_n(p(a_n,z_{1n},\dots,z_{kn}))
=\varphi(p(a,\zeta_11,\dots,\zeta_k1))
\end{equation}
for all non-commutative polynomials $p$.
Moreover, by virtue of the linearity of a state we need only to prove 
\eqref{1eqn:in proof 1.1 (100)} for all non-commutative monomials of the form:
\begin{equation}\label{1eqn:in proof 1.1 (101)}
p(x,y_1,\dots,y_k)=
x^{\alpha_1}y_1^{\beta_1}\dotsb y_k^{\delta_1}\dotsb 
x^{\alpha_i}y_1^{\beta_i}\dotsb y_k^{\delta_i}\dotsb.
\end{equation}
We will prove this by induction on the degree $m$ of the monomial. 
Here the degree of $p$ in \eqref{1eqn:in proof 1.1 (101)} is defined by
\[
m=\sum_i (\alpha_i+\beta_i+\dotsb+\delta_i).
\]

For $m=1$ we need to show that
\[
\lim_{n\rightarrow\infty}
\varphi_n(a_n)=\varphi(a),
\qquad
\lim_{n\rightarrow\infty}
\varphi_n(z_{in})
=\varphi(\zeta_i1).
\]
But these are obvious by assumption (i) of $a_n \mconv a$ and $z_{in}\mconv \zeta_i 1$.
Let $m\ge1$ and suppose that 
\eqref{1eqn:in proof 1.1 (100)} is true for all non-commutative monomials 
\eqref{1eqn:in proof 1.1 (101)} of degree up to $m$.
Now let $p(x,y_1,\dots,y_k)$ be a non-commutative monomial of degree $m+1$.
We need to prove \eqref{1eqn:in proof 1.1 (100)} for this monomial.

(Case 1) $p(x,y_1,\dots,y_k)=x^{m+1}$.
In this case \eqref{1eqn:in proof 1.1 (100)} holds obviously by the assumption of $a_n \mconv a$.

(Case 2) $p(x,y_1,\dots,y_k)=x^\alpha y_i q$ 
with $\alpha\ge 0$ and $q=q(x,y_1,\dots,y_k)$ being a non-commutative 
monomial of degree $m-\alpha$.
For simplicity we set
\[
p(a_n,z_{1n},\dots,z_{kn})=a_n^\alpha z_{in}w_n\,,
\qquad
p(a,\zeta_11,\dots,\zeta_k1)=a^\alpha\zeta_iw.
\]
Then we have
\begin{align}
&|\varphi_n(a_n^\alpha z_{in} w_n)-\varphi(a^\alpha \zeta_i w)| 
\nonumber \\
&\qquad\le |\varphi_n(a_n^\alpha z_{in} w_n)-\varphi_n(a_n^\alpha \zeta_i w_n)|
+|\varphi_n(a_n^\alpha \zeta_i w_n)-\varphi(a^\alpha \zeta_i w)| 
\nonumber \\
&\qquad=|\varphi_n(a_n^\alpha (z_{in}-\zeta_i1) w_n)|
+|\zeta_i||\varphi_n(a_n^\alpha w_n)-\varphi(a^\alpha w)| 
\nonumber \\
&\qquad=|\varphi_n( (z_{in}-\zeta_i1) w_na_n^\alpha)|
+|\zeta_i||\varphi_n(a_n^\alpha w_n)-\varphi(a^\alpha w)|,
\label{1eqn:in proof Prop 1.2 (23)}
\end{align}
where the last identity is due to assumption (ii).
The second term of \eqref{1eqn:in proof Prop 1.2 (23)} 
tends to $0$ as $n\rightarrow\infty$ by the assumption of induction.
For the first term we apply the Schwarz inequality to obtain
\begin{equation}\label{1eqn:in proof Prop 1.2 (21)}
|\varphi_n( (z_{in}-\zeta_i1) w_na_n^\alpha)|^2
\le \varphi_n((z_{in}-\zeta_i1)^2)\varphi((w_na_n^\alpha)^*(w_na_n^\alpha)). 
\end{equation}
Since $(w_na_n^\alpha)^*(w_na_n^\alpha)$ is a monomial of degree $2m$, 
we have $\varphi((w_na_n^\alpha)^*(w_na_n^\alpha))\le C_{2m}$ 
by the uniformly bounded assumption (iii).
Then \eqref{1eqn:in proof Prop 1.2 (21)} becomes
\[
|\varphi_n( (z_{in}-\zeta_i1) w_na_n^\alpha)|^2
\le C_{2m}  \{\varphi_n(z_{in}^2)-2\zeta_i\varphi_n(z_{in})+\zeta_i^2\} 
\rightarrow0
\quad \text{as $n\rightarrow\infty$}.
\]
Hence \eqref{1eqn:in proof Prop 1.2 (23)} tends to $0$  as $n\rightarrow\infty$.
Consequently,
\[
\lim_{n\rightarrow\infty}\varphi_n(a_n^\alpha z_{in} w_n)=\varphi(a^\alpha \zeta_i w).
\]
Thus, \eqref{1eqn:in proof 1.1 (100)} holds for our monomial $p(x,y_1,\dots,y_k)=x^\alpha y_i q$. 

Finally, we see from (Case 1) and (Case 2) that \eqref{1eqn:in proof 1.1 (100)} is true 
also for all non-commutative monomials $p(x,y_1,\dots,y_k)$ of degree $m+1$.
This completes the proof.
\end{proof}

\subsection{Adjacency Matrix as Algebraic Random Variable}
\label{subsec:Adjacenct Matrix as Algebraic Random Variable}

Let $G=(V,E)$ be a finite graph and $A$ the adjacency matrix.
Let $\mathcal{A}(G)$ be the adjacency algebra, i.e.,
the $*$-algebra generated by $A$.
Define the normalized trace by
\begin{equation}\label{2eqn:normalized trace}
\varphi_{\mathrm{tr}}(a)=\frac{1}{|V|}\,\mathrm{Tr\,} a,
\qquad a\in \mathcal{A}(G).
\end{equation}
Then, $\varphi_{\mathrm{tr}}$ becomes a state on $\mathcal{A}(G)$
and the adjacency matrix $A$ is regarded as a real random variable
of the algebraic probability space $(\mathcal{A}(G),\varphi_{\mathrm{tr}})$.

\begin{proposition}
The spectral distribution of the adjacency matrix $A$ in the state $\varphi_{\mathrm{tr}}$
coincides with the eigenvalue distribution of the graph $G$.
In other words, it holds that
\[
\varphi_{\mathrm{tr}}(A^m)
=\int_{-\infty}^{+\infty} x^m \mu(dx),
\qquad m=1,2,\dots,
\]
where $\mu$ is the eigenvalue distribution of $G$.
\end{proposition}

The proof is obvious; however, the above relation is a clue to study
the eigenvalue distribution of a graph by means of quantum probabilistic techniques.

\subsection{Main Result}
\label{subsec:Main Result}

Let $G=(V,E)$ be a graph.
For an integer $k\ge1$ the \textit{distance $k$-graph} of $G$ is a graph
$G^{[k]}=(V,E^{[k]})$ with
\[
E^{[k]}=\{\{x,y\}\,;\, x,y\in V, \, \partial_G(x,y)=k\},
\]
where $\partial_G(x,y)$ is the graph distance of $G$.
The distance $1$-graph of $G$ coincides with $G$ itself.

Now we rephrase the main result.
Let $G=(V,E)$ be a finite connected graph with $|V|\ge2$.
For $k\ge1$ and $N\ge1$ let $G^{[N,k]}$ be the distance $k$-graph 
of $G^N=G\times \dotsb\times G$ ($N$-fold Cartesian power).
In general, $G^{[N,k]}$ is not necessarily connected.
The adjacency matrix $A^{[N,k]}$ of $G^{[N,k]}$ is considered 
as a real random variable of the algebraic probability space
$(\mathcal{A}(G^{[N,k]}),\varphi_{\mathrm{tr}})$,
where $\varphi_{\mathrm{tr}}$ is the normalized trace, 
see \ref{subsec:Adjacenct Matrix as Algebraic Random Variable}.
The main result (Theorem \ref{1thm:main result}) is equivalent to the following statement.

\begin{theorem}\label{3thm:main result}
Notations and assumptions being as above,
we have
\[
\frac{A^{[N,k]}}{N^{k/2}}
\mconv \left(\frac{2|E|}{|V|}\right)^{k/2}
\frac{1}{k!}\Tilde{H}_k(g),
\]
where $\Tilde{H}_k$ is the monic Hermite polynomial (see below)
and $g$ is a random variable obeying the standard normal distribution $N(0,1)$.
\end{theorem}

After the standard terminology (e.g., \cite{Beals-Wong,Chihara}) 
the \textit{Hermite polynomials} $\{H_n(x)\}$ are defined by
the three-term recurrence relation:
\begin{align*}
H_0(x)&=1, \\
H_1(x)&=2x, \\
2xH_n(x)&=H_{n+1}(x)+2nH_{n-1}(x).
\end{align*}
The \textit{monic Hermite polynomials} appeared in Theorem \ref{3thm:main result}
are defined after a simple normalization:
\[
\Tilde{H}_n(x)=2^{-n/2}H_n\left(\frac{x}{\sqrt2}\right),
\qquad n=0,1,2, \dots.
\]
Then we have
\begin{align}
\Tilde{H}_0(x)&=1, \nonumber \\
\Tilde{H}_1(x)&=x, \nonumber \\
x\Tilde{H}_n(x)&=\Tilde{H}_{n+1}(x)+n\Tilde{H}_{n-1}(x).
\label{2eqn:recurrence for Hermite polynomials}
\end{align}
It is known that $\{\Tilde{H}_n(x)\}$ becomes the orthogonal polynomials 
with respect to the standard normal distribution $N(0,1)$. 
We remark that they are not normalized to have norm one; in fact,
\[
\frac{1}{\sqrt{2\pi}}\int_{-\infty}^{+\infty} \Tilde{H}_n(x)^2 e^{-x^2/2}dx= n!\,,
\qquad n=0,1,2, \dots.
\]

\section{Convergence of Tensor Powers of Algebraic Random Variables}
\label{2sec:Convergence}

Let $(\mathcal{A},\varphi)$ be an arbitrary algebraic probability space.
For $N\ge1$ we consider the $N$-fold tensor power $(\mathcal{A}^{\otimes N},\varphi^{\otimes N})$.
From now on we write $\varphi$ for $\varphi^{\otimes N}$.
For a real random variable $b=b^*\in \mathcal{A}$ and
$i\in\{1,2,\dots,N\}$ we define $b(i)\in\mathcal{A}^{\otimes N}$ by
\[
b(i)=
\overbrace{\mathstrut 1\otimes \dots\otimes 1 \otimes  b \otimes 1\otimes \dots \otimes 1}
^{\text{$N$ factors}},
\]
where $b$ appears at the $i$-th position.
Let $\mathcal{B}_N$ denote the $*$-algebra generated by
$b(1),b(2),\dots,b(N)$.
Obviously, $\mathcal{B}_N$ becomes a commutative $*$-subalgebra of $\mathcal{A}^{\otimes N}$.
For mutually distinct $i_1,\dots,i_n\in\{1,2,\dots,N\}$ we 
define $b(i_1,\dots,i_n)\in\mathcal{B}_N$ by
\begin{align*}
b(i_1,\dots,i_n)
&=b(i_1)\dotsb b(i_n) \\
&=1\otimes \dots\otimes 1 \otimes  b \otimes
1\otimes \dots \otimes 1 \otimes b \otimes 1 \otimes \dotsb \otimes 1,
\end{align*}
where $b$ appears at $i_1$-th,$\dots$, $i_n$-th positions.
Finally for $1\le n\le N$ we set
\begin{equation}\label{3eqn:def of b(N,n)}
b^{(N,n)}
=\sum_{1\le i_1<\dots<i_n\le N} b(i_1,\dots, i_n)
=\frac{1}{n!} \sum_{\substack{i_1,\dots,i_n \\ \neq}} b(i_1,\dots, i_n)
\end{equation}
and for convenience
\[
b^{(N,0)}=1\otimes\dots\otimes 1.
\]

We are interested in the asymptotic spectral distribution of $b^{(N,n)}$ as $N\rightarrow\infty$.
For $n=1$ the result is well known, see e.g., \cite[Chapter 8]{HO-BOOK}.

\begin{theorem}[Commutative law of large numbers]
\label{3thm:CLLN}
For a real random variable $b=b^*\in\mathcal{A}$ we have
\begin{equation}\label{3eqn:CLLN}
\frac{b^{(N,1)}}{N}\mconv \varphi(b)
\quad \text{as $N\rightarrow\infty$}.
\end{equation}
\end{theorem}

\begin{theorem}[Commutative central limit theorem]
\label{3thm:CLT0}
For a real random variable $b=b^*\in\mathcal{A}$ with $\varphi(b)=0$ and $\varphi(b^2)=1$
we have
\begin{equation}\label{3eqn:CLT}
\frac{b^{(N,1)}}{\sqrt{N}}\mconv g
\quad \text{as $N\rightarrow\infty$},
\end{equation}
where $g$ is a Gaussian random variable obeying the standard normal law $N(0,1)$.
\end{theorem}

The commutative independence of $b(1),b(2),\dots,b(N)$ is essential
in the above statements.
We recall that \eqref{3eqn:CLT} means that
\begin{align*}
\lim_{N\rightarrow\infty}\varphi \bigg(\bigg(\frac{b^{(N,1)}}{\sqrt{N}}\bigg)^m\bigg)
&=\text{($m$-th moment of $N(0,1)$)} \\
&=\begin{cases}
 0, & \text{$m$: odd}, \\
 \dfrac{(2k)!}{2^kk!}\,, & \text{$m=2k$: even},
 \end{cases}
\quad m=1,2,\dots.
\end{align*}
We are now in a position to state a generalization of Theorem \ref{3thm:CLT0}.

\begin{theorem}\label{2thm:convergence of B[n]}
Notations and assumptions being as in Theorem \ref{3thm:CLT0}, we have
\[
\frac{b^{(N,n)}}{N^{n/2}}\mconv \frac{1}{n!}\Tilde{H}_n(g)
\quad \text{as $N\rightarrow\infty$}
\]
for all $n=1,2,\dots$,
where $\Tilde{H}_n$ is the monic Hermite polynomial of degree $n$ defined in
Section \ref{subsec:Main Result}.
\end{theorem}

Before going into the proof, we observe the case of $n=2$ in detail for grasping the situation.
We keep in mind that $b=b^*\in\mathcal{A}$ with $\varphi(b)=0$ and $\varphi(b^2)=1$.
Starting with the simple identities:
\begin{align*}
b^{(N,1)}b^{(N,1)}
&=\Bigg(\sum_{1\le i\le N} b(i)\Bigg)\Bigg(\sum_{1\le i\le N} b(i)\Bigg)
=\sum_{\substack{i_1,i_2 \\ \neq}} b(i_1,i_2)+\sum_{1\le i\le N} b^2(i) \\
&=2b^{(N,2)}+\sum_{1\le i\le N} b^2(i)
=2b^{(N,2)}+N+\sum_{1\le i\le N} (b^2-1)(i),
\end{align*}
we obtain
\begin{equation}\label{2eqn:B[2] observation}
2\,\frac{b^{(N,2)}}{N}
=\frac{b^{(N,1)}}{\sqrt{N}}\,\frac{b^{(N,1)}}{\sqrt{N}} 
 -1-\frac{1}{N}\sum_{1\le i\le N} (b^2-1)(i).
\end{equation}
For simplicity we set
\begin{equation}\label{3eqn:aN and z1N}
a_N=\frac{b^{(N,1)}}{\sqrt{N}}\,,
\qquad
z_{1N}=\frac{1}{N}\sum_{1\le i\le N} (b^2-1)(i).
\end{equation}
Then \eqref{2eqn:B[2] observation} becomes
\begin{equation}\label{3eqn:B[2] observation2}
2\,\frac{b^{(N,2)}}{N}
=a_N^2 -1- z_{1N}\,.
\end{equation}
Moreover, $a_N$ and $z_{1N}$ are commutative, and 
\[
a_N \mconv g,
\qquad
z_{1N} \mconv 0,
\]
which follow from Theorem \ref{3thm:CLT0} and Theorem \ref{3thm:CLLN}, respectively.
Noting that \eqref{3eqn:B[2] observation2} is a polynomial in $a_N$ and $z_{1N}$,
we apply Proposition~\ref{1prop:basic convergence result} to obtain
\[
2\,\frac{b^{(N,2)}}{N}
\mconv g^2-1=\Tilde{H}_2(g),
\]
which proves Theorem~\ref{2thm:convergence of B[n]} in case of $n=2$.

\begin{proof}[Proof of Theorem \ref{2thm:convergence of B[n]}]
We need notation.
For $1\le n\le N$ and $y\in\mathcal{A}$ we define $F^{(N,n)}(y)\in\mathcal{A}^{\otimes N}$ by
\begin{equation}\label{3eqn:FNn}
F^{(N,n)}(y)
=\sum \overbrace{\mathstrut 1\otimes \dots \otimes b \otimes \dots \otimes
 y\otimes \dots \otimes b \otimes \dotsb \otimes 1}^{\text{$N$ factors}},
\end{equation}
where $b$ appears $n-1$ times and $y$ just once, 
and the sum is taken over all possible arrangements.
Then, after simple calculation we obtain
\begin{equation}\label{1lem:preliminary identity}
b^{(N,1)}b^{(N,n)}=(n+1)b^{(N,n+1)}+(N-n+1)b^{(N,n-1)}+F^{(N,n)}(b^2-1),
\end{equation}
where $1\le n<N$.

For simplicity we set
\begin{equation}\label{3eqn:def of B_n}
B_{nN}=n!\,\frac{b^{(N,n)}}{N^{n/2}}\,,
\quad
z_{nN}=\frac{F^{(N,n)}(b^2-1)}{N^{(n+1)/2}}\,,
\quad 1\le n\le N.
\end{equation}
Obviously, these are members of $\mathcal{B}_N$.
For $n=1$ we have $B_{1N}=a_N$, see also \eqref{3eqn:aN and z1N}.
With these notations \eqref{1lem:preliminary identity} becomes
\begin{equation}\label{2eqn: in proof lemma 2.5 (11)}
B_{n+1,N}=a_N B_{nN}-n B_{n-1,N}
+\frac{n(n-1)}{N}\,B_{n-1,N}-n!z_{nN}\,.
\end{equation}
We will show that
for each $n=1,2,\dots$ there exist a polynomial $p_n(x,y_1,\dots,y_{n-1})$ 
(independent of $N$) such that
\begin{equation}\label{2eqn:preliminary relation}
Y_{nN}\equiv B_{nN}-p_n(a_N\,, z_{1N}\,, z_{2N}\,, \dots, z_{n-1,N})
\end{equation}
is a real random variable in $\mathcal{B}_N$ and
$Y_{nN}\mconv 0$ as $N\rightarrow\infty$.
The assertion for $n=1$ is trivial with 
\begin{equation}\label{3eqn:n=1}
p_1(x)=x,
\qquad
Y_{1N}=0.
\end{equation}
For $n=2$ we see from \eqref{3eqn:B[2] observation2} that
\begin{equation}\label{3eqn:n=2}
p_2(x,y_1)=x^2-1-y_1\,,
\qquad
Y_{2N}=0.
\end{equation}
Suppose that the assertion holds up to $n\ge2$.
Then \eqref{2eqn: in proof lemma 2.5 (11)} becomes
\begin{align}
B_{n+1,N}
&=a_N (p_n+Y_{nN})-n (p_{n-1}+Y_{n-1,N}) \nonumber \\
&\qquad +\frac{n(n-1)}{N}\,(p_{n-1}+Y_{n-1,N})-n!z_{nN} \nonumber\\
&=a_Np_n - n p_{n-1}-n! z_{nN} \nonumber \\
&\qquad+a_NY_{nN}-\bigg\{n-\frac{n(n-1)}{N}\bigg\} Y_{n-1,N} 
+ \frac{n(n-1)}{N}\,p_{n-1}\,.
\label{3eqn:recurrence before limit}
\end{align} 
Hence, setting
\begin{align}
&p_{n+1}(x, y_1,y_2,\dots, y_n)
=x p_n(x, y_1,y_2,\dots, y_{n-1}) \nonumber \\ 
&\qquad\qquad\qquad\qquad\qquad\qquad - n p_{n-1}(x, y_1,y_2,\dots, y_{n-2}) -n! \,y_n\,, 
\label{3eqn:recurrence for p_n}\\
&Y_{n+1,N}
=a_NY_{nN}-\bigg\{n-\frac{n(n-1)}{N}\bigg\} Y_{n-1,N} \nonumber\\
&\qquad\qquad\qquad + \frac{n(n-1)}{N}\,p_{n-1}(a_N\,, z_{1N}\,, z_{2N}\,, \dots, z_{n-2,N}),
\label{3eqn:recurrence for Y_n}
\end{align}
we have
\begin{equation}\label{3eqn:recurrence before limit2}
B_{n+1,N}
=p_{n+1}(a_N\,, z_{1N}\,, z_{2N}\,, \dots, z_{nN})+Y_{n+1,N}\,.
\end{equation}
It is clear that $p_{n+1}(x, y_1,y_2,\dots, y_n)$ is a polynomial
and that $Y_{n+1,N}$ is a real random variable in $\mathcal{B}_N$.
In \eqref{3eqn:recurrence for Y_n} we have 
\[
p_{n-1}(a_N\,, z_{1N}\,, z_{2N}\,, \dots, z_{n-2,N})
\mconv p_{n-1}(g,0,0,\dots,0),
\]
which follows by Proposition~\ref{1prop:basic convergence result} 
and the fact that $z_{nN}\mconv 0$ as $N\rightarrow\infty$ (Lemma \ref{3lem:znN tends to 0} below).
Hence, applying the assumption of induction, we see that $Y_{n+1,N}\mconv 0$.
This completes the induction.

Finally, applying Proposition~\ref{1prop:basic convergence result} to
\eqref{2eqn:preliminary relation}, we obtain
\begin{equation}\label{3eqn:limit 101}
B_{nN}\mconv p_n(g, 0,0,\dots,0),
\qquad n=1,2,\dots.
\end{equation}
On the other hand, we know from \eqref{3eqn:n=1} and \eqref{3eqn:n=2} that
\[
p_1(x)=x,
\qquad
p_2(x,0)=x^2-1.
\]
Moreover, from \eqref{3eqn:recurrence for p_n} we have
\[
p_{n+1}(x,0,\dots,0)=xp_n(x,0,\dots,0)- n p_{n-1}(x,0,\dots,0).
\]
Comparing with the recurrence relation 
\eqref{2eqn:recurrence for Hermite polynomials} satisfied by the monic Hermite polynomials,
we see that
\[
p_n(x,0,\dots,0)=\Tilde{H}_n(x).
\]
Consequently, it follows from \eqref{3eqn:limit 101} that
\[
B_{nN}\mconv \Tilde{H}_n(g),
\qquad n=1,2,\dots,
\]
which completes the proof.
\end{proof}

\begin{lemma}\label{3lem:znN tends to 0}
For $n=1,2,\dots$ we have
\[
z_{nN}=\frac{F^{(N,n)}(b^2-1)}{N^{(n+1)/2}} \mconv 0.
\]
\end{lemma}

\begin{proof}
We need to show that $\varphi(z_{nN}^m)\rightarrow 0$ as $N\rightarrow \infty$ for 
fixed $m,n=1,2,\dots$.
For $m=1$ the assertion is obvious so we assume that $m\ge2$.
For simplicity we set $z=b^2-1$.
By definition we have
\[
F^{(N,n)}(z)
=\sum 1\otimes \dots \otimes b \otimes \dots \otimes
 z\otimes \dots \otimes b \otimes \dotsb \otimes 1,
\]
where $b$ appears $n-1$ times and $z$ just once, 
and the sum is taken over all possible arrangements.
Then $\varphi[(F^{(N,n)}(z))^m]$ is the sum of terms of the form
\begin{equation}\label{2eqn:in proof Lemma 2.5 (1)}
\varphi(1\otimes \dots \otimes (*) \otimes \dots \otimes
 (*)\otimes \dots \otimes (*) \otimes \dotsb \otimes 1),
\end{equation}
where $(*)$ is of the form $b^sz^t$ with $1\le s+t\le m$.
If one of the $(*)$'s is occupied by $b$ or $z$ (i.e., $s+t=1$),
the value of \eqref{2eqn:in proof Lemma 2.5 (1)} is zero since $\varphi(b)=\varphi(z)=0$.
Hence $\varphi[(F^{(N,n)}(z))^m]$ is the sum of the terms
\eqref{2eqn:in proof Lemma 2.5 (1)} such that
$(*)$ is of the form $b^sz^t$ with $2\le s+t\le m$.
We divide the sum into two parts.
Let $S$ be the sum of terms \eqref{2eqn:in proof Lemma 2.5 (1)} with $(*)$ being
of order 2, i.e., $b^2$, $bz$ or $z^2$. This happens only when $nm$ is even.
We write 
\[
\varphi[(F^{(N,n)}(z))^m]=S+R.
\]
For the estimate we set 
\[
K=K_m=\max\{1,|\varphi(b^sz^t)|\,;\, 2\le s+t\le m\}.
\]
First each term constituting $S$ is estimated as
\[
|\varphi(1\otimes \dots \otimes (*) \otimes \dots \otimes
 (*)\otimes \dots \otimes (*) \otimes \dotsb \otimes 1)|
\le K^{nm/2}.
\]
We need to count the number of such terms.
The number of choice of places where $(*)$ appears 
is given by $\binom{N}{nm/2}$.
Then the arrangements of $b^2, bz, z^2$ at a set of chosen places $(*)$ is a finite number $c_1(m,n)$
depending on $m$ and $n$, though the explicit expression is not simple.
Hence
\[
S\le K^{nm/2} \binom{N}{nm/2} c_1(m,n)\le C_1(m,n) N^{nm/2}
\]
for some constant $C_1(m,n)$.
If  $nm$ is odd, letting $S$ be the sum of terms \eqref{2eqn:in proof Lemma 2.5 (1)} with $(*)$ being
of order 2 except one $(*)$ of order 3,
we have
\[
S\le K^{(nm-1)/2} \binom{N}{(nm-1)/2} c_2(m,n)\le C_2(m,n) N^{(nm-1)/2}
\]
In any case we have
\[
S=O(N^{[nm/2]}).
\]
By a similar argument we see easily that the rest term $R$ has a smaller order:
\[
R=o(N^{[nm/2]}).
\]
Consequently, we have
\[
\varphi(z_{n,N}^m)
=\varphi\bigg(\bigg(\frac{F^{(N,n)}(z)}{N^{(n+1)/2}}\bigg)^m\bigg)
\le \frac{S+R}{N^{(n+1)m/2}}
=O(N^{-m/2}),
\]
which tends to zero as $N\rightarrow\infty$.
This completes the proof.
\end{proof}

\section{Proof of Theorem~\ref{3thm:main result}}
\label{sec:distance k-graphs}

Associated with the finite graph $G=(V,E)$, we consider the full matrix algebra $\mathcal{M}(V)$,
that is the $*$-algebra of matrices with index set $V\times V$.
The adjacency algebra $\mathcal{A}(G)$ is a $*$-subalgebra of $\mathcal{M}(V)$.
The normalized trace $\varphi_{\mathrm{tr}}$ on $\mathcal{A}(G)$
defined in \eqref{2eqn:normalized trace} is naturally extended to 
$\mathcal{M}(V)$ and is denoted by the same symbol.
As $\mathcal{A}(G^{[N,k]})$ is a $*$-subalgebra of $\mathcal{M}(V)^{\otimes N}$,
the normalized trace on $\mathcal{A}(G^{[N,k]})$ coincides with
the restriction of the product state $\varphi_{\mathrm{tr}}^{\otimes N}$ on $\mathcal{M}(V)^{\otimes N}$,
which is denoted by $\varphi$ for simplicity hereafter.

Let $A$ and $A^{[N,k]}$ be the adjacency matrices of $G$
and $G^{[N,k]}$, respectively.
Following the notation in \eqref{3eqn:def of b(N,n)} we set
\[
A^{(N,n)}
=\sum_{1\le i_1<\dots<i_n\le N} A(i_1,\dots, i_n)
=\frac{1}{n!} \sum_{\substack{i_1,\dots,i_n \\ \neq}} A(i_1,\dots, i_n)
\]
and define a real random variable $C(N,k)$ by
\begin{equation}\label{3eqn:decomposition}
A^{[N,k]}=A^{(N,k)}+C(N,k).
\end{equation}
To our goal we will first show that
\begin{equation}\label{3eqn:aim}
\frac{C(N,k)}{N^{k/2}}\mconv 0
\quad\text{as $N\rightarrow\infty$}.
\end{equation}

\begin{remark}
As is easily seen, the adjacency matrix of $G^N$ is given by
\[
A^{(N,1)}
=\sum_{i=1}^N A(i)
=\sum_{i=1}^N 1\otimes\dotsb\otimes A\otimes\dotsb\otimes 1,
\]
where $A$ sits at the $i$-th position.
Therefore, we have
\[
A^{[N,1]}=A^{(N,1)}.
\]
However, for $k\ge2$, $A^{[N,k]}=A^{(N,k)}$ does not hold in general.
While, it is easily verified that 
$A^{[N,k]}=A^{(N,k)}$ holds when $G$ is a complete graph and $1\le k\le N$.
The $N$-fold Cartesian power of the complete graph $K_d$ ($d$ stands for the
number of vertices) is called a \textit{Hamming graph} and is denoted by $H(N,d)$.
The eigenvalue distribution of the distance $k$-graph of $H(N,d)$
is obtained by means of the Krawtchouk polynomials, see \cite{Obata2012} for $d=2$
and \cite{Hibino} for an arbitrary $d$.
\end{remark}

\begin{lemma}\label{lem3.4}
Let $G=(V,E)$ be a finite connected graph and
the distance between two vertices $\xi,\eta\in V$ is denoted by $\partial_G(\xi,\eta)$.
Then for $N$-fold Cartesian power $G^N$ we have
\[
\partial_{G^N}(x,y)=\sum_{i=1}^N \partial_G(\xi_i,\eta_i),
\]
where $x=(\xi_1,\dots,\xi_N), y=(\eta_1,\dots,\eta_N)\in V^N$.
\end{lemma}

\begin{proof}
Straightforward.
\end{proof}

It is convenient to introduce the distance matrix of $G$.
For $k=1,2,\dots$ let $D^{[k]}$ be the $k$-distance matrix of $G$,
which is a matrix indexed by $V\times V$ and defined by
\[
(D^{[k]})_{xy}
=\begin{cases}
1, & \partial_G(x,y)=k, \\
0, & \text{otherwise}.
\end{cases}
\]
In other words, $D^{[k]}$ is the adjacency matrix of the distance $k$-graph of $G$.
By definition $A=D^{[1]}$ and $A^{[1,k]}=D^{[k]}$.

We need a concise expression for $C(N,k)$.
For illustration we consider the case of $k=2$.
For two vertices $x=(\xi_1,\dots,\xi_N), y=(\eta_1,\dots,\eta_N)\in V^N$, 
\[
\partial_{G^N}(x,y)
=\sum_{i=1}^N \partial_G(\xi_i,\eta_i)=2
\]
holds if and only if one of the following two cases occurs:
\begin{enumerate}
\item[(i)] there exist $1<i_1<i_2\le N$ such that 
$\partial_G(\xi_{i_1},\eta_{i_1})=\partial_G(\xi_{i_2},\eta_{i_2})=1$
and $\partial_G(\xi_{j},\eta_{j})=0$ for all $j\neq i_1,i_2$;
\item[(ii)] there exists $1\le i\le N$ such that
$\partial_G(\xi_i,\eta_i)=2$
and $\partial_G(\xi_{j},\eta_{j})=0$ for all $j\neq i$.
\end{enumerate}
We then have
\[
A^{[N,2]}
=\sum_{1\le i_1<i_2\le N} D^{[1]}(i_1,i_2) +\sum_{1\le i\le N} D^{[2]}(i)
=A^{(N,2)}+C(N,2).
\]

The above argument is applied to $A^{[N,k]}$ for a general $k$.
For $k\ge1$ we set 
\[
\Lambda(k)
=\left\{\lambda=(j_1,j_2,\dots)\,;\,
\text{$j_h\ge0$ are integers such that $\displaystyle\sum_{h=1}^\infty hj_h=k$}
\right\}.
\]
An element of $\Lambda(k)$ is called a \textit{partition} of $k$.
For $\lambda=(j_1,j_2,\dots)\in\Lambda(k)$ we define
\[
C(\lambda)=\sum 1\otimes \dotsb\otimes (*) \otimes \dotsb \otimes (*) \otimes \dotsb \otimes 1,
\]
where $(*)$ is occupied by $D^{[h]}$ with $j_h$-times ($h=1,2,\dots$)
and the sum is taken over all possible arrangements.
For $\lambda_0=(k,0,0,\dots)$, we have
\[
C(\lambda_0)=\sum_{1\le i_1<\dots<i_k\le N} D^{[1]}(i_1,\dots,i_k) =A^{(N,k)}.
\]

\begin{lemma}\label{3lem:expression of A}
For $k\ge1$, we have
\begin{equation}\label{3eqn:decomposition for k}
A^{[N,k]}
=A^{(N,k)}
+\sum_{\lambda\in\Lambda(k)\backslash\{\lambda_0\}} C(\lambda).
\end{equation}
\end{lemma}

\begin{proof}
For $x=(\xi_1,\dots,\xi_N)$ and $y=(\eta_1,\dots,\eta_N)\in V^N$,
we have $(A^{[N,k]})_{xy}=1$, i.e., $\partial_{G^N}(x,y)=k$ if and only if 
\[
\sum_{i=1}^N \partial_G(\xi_i,\eta_i)=k.
\]
Then by counting the number of pairs $(\xi_i,\eta_i)$ having the same distance $h$,
we come to \eqref{3eqn:decomposition for k} with no difficulty.
\end{proof}

\begin{lemma}\label{lem:estimate C(lambda)}
For $\lambda\in \Lambda(k)\backslash \{\lambda_0\}$ we have
\[
\frac{C(\lambda)}{N^{k/2}} \mconv 0.
\]
\end{lemma}

\begin{proof}
Let $\lambda=(j_1,j_2,\dots)$ and $J=\sum j_h$.
Let $M(m)$ be the maximum of the absolute value of the 
mixed moments of $D^{[1]},D^{[2]},\dots$ of degree $\le m$.
By explicit expansion 
\[
C(\lambda)^m=\sum 1\otimes\dots\otimes(*)\otimes\dotsb\otimes(*)\otimes\dotsb\otimes 1,
\]
where $(*)$ is a non-commutative monomial in $D^{[1]},D^{[2]},\dots$ of degree at most $m$.
On computing the value $\varphi(C(\lambda)^m)$, the terms having a monomial $D^{[h]}$ 
of degree 1 do not contribute since $\varphi(D^{[h]})=0$.
Hence we need to consider only the terms where every $(*)$ is a monomial of degree at least 2.
We write
\[
\varphi(C(\lambda)^m)
=S+R,
\]
where
\[
S=\sum \varphi(1\otimes\dots\otimes(*)\otimes\dotsb\otimes(*)\otimes\dotsb\otimes 1),
\]
where all $(*)$'s are monomials of degree 2 or
all $(*)$'s are monomials of degree 2 except one which is of degree 3
according to the parity of $mJ$.
For $S$ we see that the number of choice of places where $(*)$ appears 
is given by $\binom{N}{[mJ/2]}$.
Then the arrangements of $D^{[1]},D^{[2]},\dots$ at the chosen places is a finite number $c(m,J)$
depending on $m$ and $J$ though the explicit expression is not simple.
Hence
\[
|S|\le M(m)^{[mJ/2]}\binom{N}{[mJ/2]} c(m,k)\le C_1(m,k) N^{[mJ/2]}
\]
for some constant $C_1(m,k)$.
Similarly, for $R$ the number of choice of places where $(*)$ appears 
is $\le\binom{N}{[mJ/2]-1}$ so that
\[
|R|= o(N^{[mJ/2]}).
\]
We note that for $\lambda\in\Lambda(k)\backslash\{\lambda_0\}$ 
\[
J=\sum_{h}j_h<\sum_{h} hj_h=k.
\]
Hence $J-k\le -1$ and 
\[
\left[\frac{mJ}{2}\right]-\frac{km}{2}\le -\frac{m}{2}.
\]
Consequently,
\[
\varphi\bigg(\bigg(\frac{C(\lambda)}{N^{k/2}}\bigg)^m\bigg)
\le \frac{C_1(m,k) N^{[mJ/2]}+ o(N^{[mJ/2]})}{N^{km/2}}
=O(N^{-m/2})\rightarrow 0.
\]
This completes the proof.
\end{proof}

\begin{proof}[Proof of Theorem~\ref{3thm:main result}]
By Lemma~\ref{3lem:expression of A} we have
\begin{align}
A^{[N,k]}
&=A^{(N,k)}+C(N,k),
\label{3eqn:decomposition for k (1)}\\
C(N,k)
&=\sum_{\lambda\in\Lambda(k)\backslash\{\lambda_0\}} C(\lambda).
\end{align}
Upon applying Theorem~\ref{2thm:convergence of B[n]} to $A^{(N,k)}$ we take normalization into account.
Note first that
\[
\varphi(A)=0,
\qquad
\varphi(A^2)=\frac{2|E|}{|V|}.
\]
In fact, $\varphi(A^2)$ is the mean degree of $G$.
Viewing that $A/\sqrt{\varphi(A^2)}$ is a normalized real random variable, we apply
Theorem~\ref{2thm:convergence of B[n]} to obtain
\[
\frac{A^{(N,k)}}{N^{k/2}\varphi(A^2)^{k/2}} 
\mconv \frac{1}{k!}\,\Tilde{H}_k(g).
\]
Therefore, 
\begin{equation}\label{3eqn:conv of B[N,k]}
\frac{A^{(N,k)}}{N^{k/2}} 
\mconv \left(\frac{2|E|}{|V|}\right)^{k/2}\frac{1}{k!}\,\Tilde{H}_k(g).
\end{equation}
On the other hand, for $C(N,k)$ in \eqref{3eqn:decomposition for k (1)} we have
\begin{equation}\label{3eqn:conv of C(N,k)}
\frac{C(N,k)}{N^{k/2}} 
\mconv0
\end{equation}
by Lemma~\ref{lem:estimate C(lambda)} and Proposition~\ref{1prop:basic convergence result}.
Finally, the assertion follows from
\eqref{3eqn:conv of B[N,k]} and \eqref{3eqn:conv of C(N,k)}
with the help of Proposition~\ref{1prop:basic convergence result} again.
\end{proof}

\begin{remark}
During the above argument we needed to restrict ourselves to the tracial states, 
although the combinatorial limit formula in Theorem~\ref{2thm:convergence of B[n]} holds
for a general state.
This restriction is reasonable to obtain the eigenvalue distribution of a graph
since the normalized trace on the adjacency algebra 
is related to the eigenvalue distribution of the graph, 
see Section~\ref{subsec:Adjacenct Matrix as Algebraic Random Variable}.
However, it is plausible that our argument is modified to cover a general case,
for example, a vector state (sometimes called a vacuum state) on the adjacency algebra.
The work is now in progress.
\end{remark}

\subsection*{Acknowledgements}
N. O. was supported by the JSPS Grant-in-Aid for Challenging Exploratory Research No. 23654046.

\end{document}